\documentclass[a4paper,12pt,reqno]{amsart}
\usepackage{latexsym,amscd,amssymb,amsmath,amsthm,comment,url}
\usepackage[dvips]{graphicx}
\makeatletter

\theoremstyle{plain}
  \newtheorem{thm}{Theorem}[section]
  \newtheorem{prop}[thm]{Proposition}
  \newtheorem{lem}[thm]{Lemma}
  \newtheorem{cor}[thm]{Corollary}
  
\theoremstyle{definition}
  \newtheorem{dfn}[thm]{Definition}

\theoremstyle{remark}
  \newtheorem{rem}[thm]{Remark}
\let\opn\operatorname 
\let\term\emph
\def\@bothmode#1{\ifmmode #1\else $#1$\fi}
\newcount\@tempchn 
\def\@chCount#1{%
   \@tempchn=0
   \@tfor\member:=#1\do{\advance\@tempchn by 1}%
}
\def\@autopr#1{%
   \@chCount{#1}%
   \ifnum\@tempchn<2 #1\else (#1)\fi
}
\let\@tempopn\relax 
\def\@opform_#1#2{\@tempopn_{#1}\@autopr{#2}}
\numberwithin{equation}{section}
\def\NN{\mathbb{N}} 
\def\ZZ{\mathbb{Z}} 
\def\RR{\mathbb{R}} 
\def\kk{\Bbbk} 
\def\m{\ideal{m}} 
\def\p{\ideal{p}} 
\let\s\sigma 
\let\t\tau 
\let\i\iota 
\def\MM{\mathcal M} 
\def\M{\mathbb M} 
\let\@tempar\relax 
\def\@seton^#1{\overset{#1}{\@tempar}}
\def\defar#1#2{\@xp\def\csname #1\endcsname{\def\@tempar{#2}\@ifnextchar^{\@seton}{\@tempar}}}
\defar{longto}{\longrightarrow} 
\defar{epito}{\twoheadrightarrow} 
\defar{monoto}{\rightarrowtail} 
\defar{embto}{\hookrightarrow} 
\def\imply{\@bothmode{\Rightarrow}} 
\def\Imply{\@bothmode{\Longrightarrow}} 
\def\iff{\@bothmode\Longleftrightarrow} 
\def\get{\@bothmode{\Leftarrow}} 
\def\Get{\@bothmode{\Longleftarrow}} 
\def\mbra#1{\{ #1\}} 
\def\set#1#2{\mbra{\,#1\mid #2\,}} 
\let\Dsum\bigoplus 
\let\tns\otimes 
\def\supp{\opn{supp}} 
\def\bc{a} 
\def\op{\mathsf{op}} 
\def\chara{\operatorname{char}} 
\def\defopn#1{%
    \@xp\def\csname #1\endcsname{%
        \def\@tempopn{\opn{\csname the#1\endcsname}}%
        \@ifnextchar_{\@opform}{\@opform_{}}%
    }%
}

\defopn{Ker} 

\defopn{Im} 

\defopn{Cok} 

\defopn{Ann} 

\defopn{Ass} 

\defopn{Spec} 

\defopn{pd} 

\defopn{Sym} 

\defopn{id} 
\def\idmap{\opn{id}} 
\def\the@init{in}
\defopn{@init}
\def\init{\@init_{\succ}}
\let\ideal\mathfrak 
\def\theE{E}
\def\E{\@ifstar{{}^*\theE}{\theE}} 
\let\defcat\defopn
\def\theMod{Mod}    \def\themod{mod}
\def\theMod{Mod}    \def\themod{mod}

\let\the@Lgr\theMod  \let\the@lgr\themod
\defcat{Mod} 
\defcat{mod} 
\defcat{Gr} 
\defcat{gr} 
\defcat{@Lgr}
\def\Lgr#1{\@Lgr_{\ZZ\MM}{#1}} 
\defcat{@lgr}
\def\lgr#1{\@lgr_{\ZZ\MM}{#1}} 
\defcat{Sq} 
\defcat{InjSq} 
\let\colimit\varinjlim 
\def\theHom{Hom}    \def\theRHom{RHom}
\def\theExt{Ext}    
\def\theD{D}
\let\@tempgrop\underline
\def\Hom{\@ifstar{\opn{\@tempgrop\theHom}}{\opn\theHom}} 
\def\RHom{\@ifstar{\opn{R\@tempgrop\theHom}}{\opn\theRHom}} 
\def\Ext{\@ifstar{\opn{\@tempgrop\theExt}}{\opn\theExt}} 
\def\uExt{\underline{\operatorname{Ext}}}

\def\uHom{\underline{\Hom}}
\def\RuHom{{\rm R}\uHom}

\def\theDcat{{\mathsf D}}
\def\Db{\theDcat^b} 
\def\@G_#1{\Gamma_{#1}}
\def\G{\@ifnextchar_{\@G}{\@G_\m}} 
\def\DD{\mathbb D} 
\def\DDD{\mathbf D}

\def\TT{\mathbf T}
\def\theHtcat{{\mathsf K}}
\def\Cb{\theHtcat^b} 

\def\for{{\mathbb U}} 

\def\depth{\operatorname{depth}}
\def\rank{\operatorname{rank}}
\def\hght{\operatorname{ht}}
\def\cpx#1{#1^{\bullet}} 
\def\theD{D} 
\def\D{\@ifstar{{}^*\!\theD}{\theD}} 
\def\Sh{\operatorname{Sh}} 
\def\cF{{\mathcal F}} 
\def\sp{\operatorname{Sp\acute{e}}} 
\def\RcHom{{\rm R}\operatorname{\mathcal Hom}} 
\def\cDx{\cpx {\mathcal D}_X} 
\def\cDu{\cpx {\mathcal D}_U} 
\def\const{\underline{\kk} } 

\def\cH{\mathcal H} 
\def\rH{\tilde{H}} 
\def\<{{\langle}}
\def\>{{\rangle}}

\def\ba{\mathbf a}
\def\bb{\mathbf b}
\def\bc{\mathbf c}
\def\be{\mathbf e}
\def\b0{\mathbf 0}
\def\11{\mathbf 1}
\def\ua{{\underline{\ba}}}
\def\ub{{\underline{\bb}}}
\def\uc{{\underline{\bc}}}
\def\sm{{\sf m}}

\makeatother
\title[Face ring of a simplicial poset]
{Dualizing complex of the face ring \\ of a simplicial poset}

\author{Kohji Yanagawa}
\thanks{Partially supported by Grant-in-Aid for Scientific Research (c) (no.19540028).}
\address{Department of Mathematics, Kansai University,
Suita 564-8680, Japan}
\email{yanagawa@ipcku.kansai-u.ac.jp}
\begin{document}
%
%
\maketitle

\begin{abstract}
A finite poset $P$ is called {\it simplicial}, if it has 
the smallest element 
$\hat{0}$, and every interval $[\hat{0},  x]$ is a boolean algebra.  
The face poset of a simplicial complex is a typical example.  
Generalizing the Stanley-Reisner ring of a simplicial complex, 
Stanley assigned the graded ring $A_P$ to $P$. 
This ring has been studied from both combinatorial and 
topological perspective. In this paper,  we will give a concise description of 
a dualizing complex of $A_P$, which has many applications.    
\end{abstract}

\section{Introduction}
All posets  (partially ordered sets) in this paper will be assumed to be finite. 
By the order given by inclusion, the power set of a finite set becomes a poset   
called a {\it boolean algebra}. 
We say a poset $P$ is {\it simplicial}, if it admits the smallest element 
$\hat{0}$, and the interval $[\hat{0},  x]:= \{ \, y \in P \mid y \leq x \, \}$ is isomorphic to 
a boolean algebra for all $x \in P$. 
For the simplicity, we denote $\rank (x)$ of $x \in P$ just by $\rho(x)$. 
If $P$ is simplicial and $\rho(x) = m$, then $[\hat{0}, x]$ is isomorphic to the boolean algebra 
$2^{\{1, \ldots, m\}}$. 

Let  $\Delta$ be a finite simplicial complex (with $\emptyset \in \Delta$). 
Its face poset (i.e., the set of the faces of $\Delta$ with the order given by inclusion) 
is a simplicial poset.   
Any simplicial poset $P$ is the face (cell) poset of a regular cell complex, 
which we denote by $\Gamma(P)$. 
For $\hat{0} \ne x \in P$, $c(x) \in \Gamma(P)$ denotes the open cell corresponds to $x$. 
Clearly, $\dim c(x) = \rho(x)-1$. 
While the closure $\overline{c(x)}$ of $c(x)$ is always a simplex, the intersection 
$\overline{c(x)} \cap \overline{c(y)}$ for $x,y \in P$ is not necessarily a simplex. 
For example, if two $d$-simplices are glued along their boundaries, then it is not a simplicial 
complex, but gives a simplicial poset.  

In the rest of the paper, $P$ is a simplicial complex. For $x,y \in P$, set 
$$[x \vee y]:= \text{the set of minimal elements of $\{ \, z \in P \mid z \geq x, y \, \}$.}$$
More generally, for $x_1, \ldots, x_m \in P$, $[x_1 \vee \cdots \vee x_m ]$  
denotes the set of minimal elements of the common upper bounds of  $x_1,  \ldots, x_m$. 

Set $\{ \, y \in P \mid \rho(y) = 1 \, \}=\{ y_1, \ldots, y_n \}$. 
For $U \subset [n] := \{ 1, \ldots, n \}$, we simply denote $[\bigvee_{i \in U}y_i]$ by $[U]$. 
Here $[\emptyset]= \{\hat 0\}$.  
If $x \in [U]$, then $\rho(x) = \# U$. For each $x \in P$, there exists a unique $U$ such that 
$x \in [U]$. Let $x,x' \in P$ with $x \geq x'$ and $\rho(x)= \rho(x')+1$, and take 
$U,U' \subset [n]$ such that $x \in [U]$ and $x' \in [U']$. 
Since $U = U' \coprod \{ i\}$ for some $i$ in this case, we can set 
$$\alpha(i,U) := \# \{ \, j \in U \mid j < i \, \} \qquad \text{and} \qquad 
\epsilon(x,x'):= (-1)^{\alpha(i,U)}.$$ 
Then $\epsilon$ gives an incidence function of the cell complex $\Gamma(P)$, that is, 
for all $x,y \in P$ with $x > y$ and $\rho(x)= \rho(y)+2$, we have 
$$\epsilon(x,z) \cdot \epsilon(z,y)+ \epsilon(x,z') \cdot \epsilon(z',y)=0,$$ 
where $\{z, z'\} = \{ \, w \in P \mid x>w>y \, \}$. 

Stanley \cite{St91} 
assigned the commutative ring $A_P$ to a simplicial poset $P$. 
For the definition, we remark that if $[ x \vee y] \ne \emptyset$ then  
$\{ \, z \in P \mid z \leq x, y \, \}$ has the largest element $x \wedge y$.   
Let $\kk$ be a field, and $S := \kk[ \, t_x \mid x \in P \, ]$ 
the polynomial ring in the variables $t_x$. Consider the ideal  
$$I_P:= (\, t_xt_y - t_{x \wedge y} \sum_{z \in [x \vee y]} t_z \ | \ x,y \in P   \, ) 
+ (\, t_{\hat{0}} -1 \, )$$
of $S$ (if $[x \vee y] = \emptyset$, we interpret that 
$t_xt_y - t_{x \wedge y} \sum_{z \in [x \vee y]} t_z =t_xt_y$), 
and set $$A_P:= S/I_P.$$ 

We denote $A_P$ just by $A$, if there is no danger of confusion. 
Clearly,  $\dim A_P = \rank P = \dim \Gamma(P) +1$. 
For a rank 1 element $y_i \in P$, set $t_i :=t_{y_i}$. 
If $\{ x \}= [U]$ for some $U \subset [n]$ with $\# U \geq 2$, then $t_x = \prod_{i \in U} t_i$ 
in $A$, and $t_x$ is a ``dummy variable".  Since $I_P$ is a homogeneous ideal under the 
grading given by $\deg(t_x) = \rho(x)$, $A$ is a graded ring. 
If $\Gamma(P)$ is a simplicial complex, then $A_P$ is 
generated by degree 1 elements, 
and coincides with the Stanley-Reisner ring of $\Gamma(P)$. 

Note that $A$ also has a $\ZZ^n$-grading such that $\deg t_i \in \NN^n$ is the $i$th unit vector. 
For each $x \in P$, the ideal $$\p_x:=(t_z \mid z \not \leq x)$$ of $A$ is 
a ($\ZZ^n$-graded) prime ideal 
with $\dim A/\p_x = \rho(x)$, since $A/\p_x \cong \kk[t_i \mid y_i \leq x]$.  

In \cite{Du}, Duval adapted classical argument on Stanley-Reisner rings for $A_P$, and 
got basic results. 
Recently, M. Masuda and his coworkers studied $A_P$ with a view from {\it toric topology}, 
since the {\it equivariant cohomology} ring of a torus manifold is of the form $A_P$ (cf. \cite{Mas,MP}). 
In this paper, we will introduce another approach.

Let $R$ be a noetherian commutative ring, $\Mod R$ the category of $R$-modules, 
and $\mod R$ its full subcategory consisting of finitely generated modules.  
The {\it dualizing complex} $\cpx D_R$ of $R$ gives 
the important duality $\RHom_R(-, \cpx D_R)$ on the bounded derived category $\Db(\mod R)$ (cf. \cite{Ha}). 
If $R$ is a (graded) local ring with the maximal ideal $\m$, 
then the (graded) Matlis dual of $H^{-i}(\cpx D_R)$ is the local cohomology $H_\m^i(R)$. 

We have a concise description of a dualizing complex $A_P$ as follows. 
This result refines Duval's computation of $H_\m^i(A)$ (\cite[Theorem~5.9]{Du}). 

\begin{thm}\label{main} 
Let $P$ be a simplicial poset with $d = \rank P$, and set $A := A_P$. 
The complex 
$$
\cpx I_A: 0 \to I^{-d}_A \to I^{-d+1}_A \to \cdots \to I^0_A \to 0, 
$$
given by 
$$
I^{-i}_A := \Dsum_{\substack{x \in P, \\ \rho(x) = i}} A/\p_x,
$$
and 
$$
\partial_{\cpx I_A}^{-i} : I_A^{-i} \supset A/\p_x \ni 1_{ A/\p_x} \longmapsto 
\sum_{\substack{\rho(y)  =   i-1, \\ y \leq x}} \epsilon(x,y) \cdot 1_{A/\p_y} \in 
\Dsum_{\substack{\rho(y) =  i-1, \\ y \leq x}} A/\p_y  \subset I_A^{-i+1}
$$
is isomorphic to a dualizing complex $\cpx D_A$ of $A$ in $\Db(\Mod A)$.  
\end{thm}

In \cite{Y}, the author defined a {\it squarefree module} over a polynomial ring, 
and many applications have been found. This idea is also useful for our study. 
In fact, regarding $A$ as a squarefree module 
over the polynomial ring $\operatorname{Sym} A_1$, 
Duval's formula of $H_\m^i(A)$ can be proved quickly (Remark~\ref{quick proof}).
Moreover, we can show that a theorem of Murai and Terai (\cite{MT}) on the $h$-vectors of 
simplicial complexes also holds for simplicial posets (Theorem~\ref{MT S_r}). 
In the present paper, we will define a squarefree module over $A$ to study the interaction 
between the topological properties of $\Gamma(P)$ and the homological properties of $A$.

The category $\Sq A$ of squarefree $A$-modules is an abelian category with enough injectives, 
and $A/\p_x$ is an injective object.  Hence  $\cpx I_A \in \Db(\Sq A)$, 
and $\DD(-):=\uHom^\bullet_A(-, \cpx I_A)$ gives a duality on 
$\Cb(\InjSq) \, (\cong \Db(\Sq A))$, where $\InjSq$ denotes 
the full subcategory of $\Sq A$ consisting of all injective objects 
(i.e., finite direct sums of copies of $A/\p_x$  for various $x \in P$). 
Via the forgetful functor $\Sq A \to \mod A$, $\DD$ coincides with 
the usual duality $\RHom_A(-, \cpx D_A)$ on $\Db(\mod A)$. 

By \cite{Y05}, we can assign a squarefree $A$-module $M$ 
the constructible sheaf $M^+$ on (the underlying space $X$ of) $\Gamma(P)$. 
In this context, the duality $\DD$ corresponds to the Poincar\'e-Verdier duality 
on the derived category of 
the constructible sheaves on $X$ up to translation as in \cite{OY,Y05}. 
In particular, the sheafification of the complex $\cpx I_A[-1]$ 
coincides with the Verdier dualizing complex of $X$ with the coefficients in $\kk$, where $[-1]$ represents the translation by $-1$. 
Using this argument, we can show the following. 
At least for the Cohen-Macaulay property, the next result has been   
shown in Duval \cite{Du}. However our proof gives new perspective.

\begin{cor}[see, Theorem~\ref{CM & Gor}] 
The Cohen-Macaulay, Gorenstein* and Buchsbaum properties,  and Serre's condition 
$(S_i)$ of $A_P$ depend only on the topology of the underlying space of $\Gamma(P)$ and $\chara(\kk)$.
Here we say $A_P$ is Gorenstein*, if $A_P$ is Gorenstein and the graded canonical 
module $\omega_{A_P}$ is generated by its degree 0 part.  
\end{cor}

While Theorem~\ref{main} and the results in \S4 are very parallel to the 
corresponding ones for {\it toric face rings} (\cite{OY}), 
the construction of a toric face ring and that of $A_P$ are not so similar. 
Both of them are generalizations of the notion of Stanley-Reisner rings, 
but the directions of the generalizations are almost opposite 
(for example, Proposition~\ref{Gorenstein} does not hold for toric face rings).        
The prototype of the results in \cite{OY} and the present paper is found in 
\cite{Y05}. However, the subject there is ``sheaves on a poset", and  
the connection to our rings is not so straightforward.

\section{The proof of Theorem~\ref{main}}
In the rest of the paper, $P$ is a simplicial poset with $\rank P=d$. 
We use the same convention as the preceding section, in particular, 
$A=A_P$, $\{ \, y \in P \mid \rho(y) = 1 \, \}=\{ y_1, \ldots, y_n \}$, and 
$t_i :=t_{y_i} \in A$. 
 
For a subset $U \subset [n]  = \{ \, 1, \ldots, n  \, \}$,  
$A_U$ denotes the localization of $A$ by the multiplicatively 
closed set $\{ \, \prod_{i \in U} t_i^{a_i} \mid a_i \geq 0 \, \}$.  
\begin{lem}\label{idempotent}
For $x \in [U]$, 
$$u_x:= \frac{t_x}{\prod_{i \in U} t_i} \in A_U$$
is an idempotent. 
Moreover, $u_x \cdot u_{x'}=0$ for $x,x' \in [U]$ with $x \ne x'$, and 
\begin{equation}\label{idempotent decomposition}
1_{A_U} = \sum_{x \in [U]}u_x.
\end{equation}
Hence we have a $\ZZ^n$-graded direct sum decomposition
$$A_U = \bigoplus_{x \in [U]} A_U \cdot u_x$$
(if $[U] = \emptyset$, then $A_U = 0$).
\end{lem}

\begin{proof}
Since ${\prod_{i \in U} t_i} = \sum_{x \in [U]}t_x$ in $A$, the equation \eqref{idempotent decomposition}
is clear. For $x,x' \in [U]$ with $x \ne x'$, we have $[x \vee x'] = \emptyset$ and $t_x \cdot t_{x'}=0$. 
Hence $u_x \cdot u_{x'}=0$ and 
$$u_x = u_x \cdot 1_{A_U} = u_x \cdot \sum_{x'' \in [U]} u_{x''}= u_x \cdot u_x.$$
Now the last assertion is clear.  
\end{proof}

Let $\Gr A$ be the category of $\ZZ^n$-graded $A$-modules, 
and  $\gr A$ its full subcategory consisting of finitely generated modules. 
Here a morphism $f :M\to N$ in $\Gr A$ is an $A$-homomorphism with 
$f(M_\ba) \subset N_\ba$ for all $\ba \in \ZZ^n$. 
As usual, for $M$ and $\ba \in \ZZ^n$, $M(\ba)$ denotes the 
shifted module of $M$ with $M(\ba)_\bb = M_{\ba + \bb}$. 
For $M, N \in \Gr A$, 
$$\uHom_A(M,N) := \Dsum_{\ba \in \ZZ^n} \Hom_{\Gr A}(M, N(\ba))$$ 
has a $\ZZ^n$-graded $A$-module structure. 
Similarly,  $\uExt_A^i(M,N) \in \Gr A$ can be defined. 
If $M \in \gr A$, the underlying module of $\uHom_A(M,N)$ 
is isomorphic to $\Hom_A(M,N)$, and the same is true for $\uExt_A^i(M,N)$. 

If $M \in \Gr A$, then $M^\vee := \Dsum_{\ba \in \ZZ^n}\Hom_\kk(M_{-\ba}, \kk)$ can be regarded as a $\ZZ^n$-graded 
$A$-module, and $(-)^\vee$ gives an exact contravariant functor from  $\Gr A$ to 
itself, which is called the {\it graded Matlis duality functor}. For $M \in \Gr A$, 
it is {\it Matlis reflexive} (i.e., $M^{\vee \vee} \cong M$) if and only if 
$\dim_\kk M_\ba < \infty$ for all $\ba \in \ZZ^n$.

\begin{lem}
$A_U \cdot u_x$ is Matlis reflexive, and $E_A(x):=(A_U \cdot u_x)^\vee$ is injective in $\Gr A$. 
Moreover, any indecomposable injective in $\Gr A$ is isomorphic to 
$E_A(x) (\ba)$ for some $x \in P$ and $\ba \in \ZZ^n$.  
\end{lem}

\begin{proof}
Clearly, $A_U \cdot u_x$ is a $\ZZ^n$-graded free $\kk[ \, t_i^{\pm 1} \mid i \in U \,]$-module. 
For $\ba=(a_1, \ldots, a_n) \in \ZZ^n$, let $\ba' \in \NN^n$ be the vector whose $i$th component is 
$$
a_i'= \begin{cases}
a_i & \text{if $i \not \in U$,}\\
1 & \text{otherwise.}
\end{cases}
$$ 
Then we have $\dim_\kk (A_U \cdot u_x)_\ba = \dim_\kk (A \cdot t_x)_{\ba'} 
\leq \dim_\kk A_{\ba'} < \infty,$
and $A_U \cdot u_x$ is Matlis reflexive. 

The injectivity of $E_A(x)$ follows from the same argument as \cite[Lemma~11.23]{MS}. 
In fact, we have a natural isomorphism
$$\uHom_A(M, E_A(x)) \cong  (M \otimes_A E_A(x)^\vee )^\vee$$ for $M \in \Gr A$ by \cite[Lemma~11.16]{MS}. 
Since $E_A(x)^\vee \, ( \cong A_U \cdot u_x)$ is a flat $A$-module,  $\uHom_A(-, E_A(x))
$ 
gives an exact functor.   

Since $E_A(x)$ is the injective envelope of $A/\p_x$  in $\Gr A$,  
and an associated prime of $M \in \Gr A$ is $\p_x$ for some $x \in P$, 
the last assertion follows. 
\end{proof}

If $(A_U \cdot u_x)_{-\ba} \ne 0$ for $\ba \in \NN^n$, then it is obvious that $\ba \in \NN^U$ 
(i.e., $a_i = 0$ for $i \not \in U$). As shown in the above proof, 
we have $\dim_\kk (A_U \cdot u_x)_{-\ba} =1$ with 
$t^{-\ba} \cdot u_x := u_x/\prod_{i \in U} t_i^{a_i} \in (A_U \cdot u_x)_{-\ba}$ 
in this case.  

For $M \in \Gr A$, its ``$\NN^n$-graded part" $M_{\geq \b0} := \bigoplus_{\ba \in \NN^n} M_\ba$ 
is a submodule of $M$. Then we have a  canonical injection 
$$\phi_x : A/\p_x \longrightarrow E_A(x)$$ defined as follows:   
The set of the monomials 
$t^\ba := \prod_{i \in U} t_i^{a_i} \in A/\p_x \cong \kk[ \, t_i \mid i \in U \, ]$ with 
$\ba \in \NN^U$ forms a $\kk$-basis of $A/\p_x$ ($\prod_{i \in U} t_i = t_x$ here), and 
$\phi_x(t^\ba) \in (E_A(x))_\ba =  \Hom_\kk( \, (A_U \cdot u_x)_{-\ba}, \, \kk \,)$ 
for $\ba \in \NN^U$ is simply given by $t^{-\ba} \cdot u_x \longmapsto 1$. 
Note that $\phi_x$ induces the isomorphism 
\begin{equation}\label{phi positive}
A/\p_x \cong E_A(x)_{\geq \b0}.
\end{equation}

\medskip

The C\v ech complex $\cpx C$ of $A$ with respect to $t_1, \ldots,  t_n$ is of the form 
$$0 \to C^0 \to C^1 \to \cdots \to C^d \to 0 \quad \text{with} \quad 
C^i = \Dsum_{\substack {U \subset [n] \\ \# U =i}}A_U$$  
(note that if $\# U > d = \dim A$ then $A_U=0$). The differential map is given by 
$$C^i \supset A_U \ni a \longmapsto \sum_{\substack {U' \supset U \\ \# U' =i+1}}
(-1)^{\alpha(U' \setminus U,U)} 
f_{U',U}(a) \in \Dsum_{\substack {U' \supset U \\ \# U' =i+1}} A_{U'} \subset C^{i+1},$$
where $f_{U',U}:A_U \to A_{U'}$ is the natural map.

Since the radical of the ideal $(t_1, \ldots, t_n)$ is the graded maximal ideal 
$\m:= (t_x \mid \hat{0}\ne x \in P )$, 
the cohomology $H^i(\cpx C)$ of $\cpx C$ is isomorphic to the local cohomology $H_\m^i(A)$. 
Moreover, $\cpx C$ is isomorphic to ${\rm R}\Gamma_\m A$ in the bounded derived category 
$\Db(\Mod A)$. Here ${\rm R}\Gamma_\m: \Db(\Mod A) \to \Db(\Mod A)$ is the right derived functor of 
$\Gamma_\m: \Mod A \to \Mod A$ given by $\Gamma_\m(M) = \{ s \in M \mid \text{$\m^i s= 0$ 
for $i \gg 0$} \, \}$.

The same is true in the $\ZZ^n$-graded context. 
We may regard $\Gamma_\m$ as a functor from $\Gr A$ to itself, and let ${}^*{\rm R} \Gamma_\m :
\Db(\Gr A) \to \Db(\Gr A)$ be its right derived functor. 
Then $\cpx C \cong {}^* {\rm R} \Gamma_\m(A)$ in $\Db(\Gr A)$.

Let $\cpx{\D*}_A$ be the $\ZZ^n$-graded normalized dualizing complex of $A$. 
By the $\ZZ^n$-graded version of the local duality theorem \cite[Theorem~V.6.2]{Ha}, 
$(\cpx{\D*}_A)^\vee \cong  {}^*{\rm R}\Gamma_\m(A)$ in $\Db(\Gr A)$. 
Since $\cpx{\D*}_A \in \Db_{\gr A}(\Gr A)$, it is Matlis reflexive, and we have
$$\cpx{\D*}_A \cong (\cpx{\D*}_A)^{\vee \vee} \cong 
{}^*{\rm R}\Gamma_\m(A)^\vee \cong (\cpx C)^\vee.$$
Since each $(C^i)^\vee$ is isomorphic to the injective object 
$$\bigoplus_{\substack{x \in P \\ \rho(x)=i}}E_A(x)$$
in $\Gr A$,  $(\cpx C)^\vee$ actually coincides with $\cpx{\D*}_A$. 
Hence $\cpx{\D*}_A$ is of the form 
$$0 \to \bigoplus_{\substack{x \in P \\ \rho(x)=d}}E_A(x) \to 
 \bigoplus_{\substack{x \in P \\ \rho(x)=d-1}}E_A(x) \to \cdots \to E_A(\hat{0}) \to 0,$$
where the cohomological degree is given by the same way as $\cpx I_A$. 

For each $i \in \ZZ$, we have an injection $\phi^i: I_A^i \to {}^*\!D^i_A$ given by 
$$I_A^i = \Dsum_{\rho(x)=-i} A/\p_x  
\supset A/\p_x \stackrel{\phi_x}{\longrightarrow} E_A(x) \subset
\Dsum_{\rho(x)=-i}E_A(x)= {}^*\!D^i_A.$$

By the definition of 
$\phi_x : A/\p_x \longrightarrow E_A(x) = (A_U \cdot u_x)^\vee$, 
we have a cochain map $$\phi^\bullet: \cpx I_A \to \cpx{\D*}_A.$$ 

\begin{lem}\label{positive}
For all $i$, the cohomology $H^i(\cpx{\D*}_A)$ of $\cpx{\D*}_A$ is $\NN^n$-graded. 
\end{lem}

This lemma immediately follows from Duval's description  of 
$H_\m^i(A)$ (\cite[Theorem~5.9]{Du}), but we give another proof using the notion of {\it squarefree modules}.  
This approach makes our proof more self-contained, and we will extend 
this idea in the following sections. 

Let $S=\kk[x_1, \ldots, x_n]$ be a polynomial ring, and regard it as a $\ZZ^n$-graded ring. 
For $\ba = (a_1, \ldots, a_n) \in \NN^n$, let $x^\ba$ denote  
the monomial $\prod x_i^{a_i} \in S$. 

\begin{dfn}[\cite{Y}]
With the above notation, a  $\ZZ^n$-graded $S$-module $M$ is called 
{\it squarefree}, if  it is finitely generated, $\NN^n$-graded 
(i.e., $M = \bigoplus_{\ba \in \NN^n} M_\ba$), and the multiplication map $M_\ba \ni t \longmapsto x_i t 
\in M_{\ba + \be_i}$ is bijective for all $\ba=(a_1, \ldots, a_n) \in \NN^n$ and all $i$ with $a_i > 0$. 
Here $\be_i \in \NN^n$ is the $i$th unit vector. 
\end{dfn}

The following lemma is easy, and we omit the proof. 

\begin{lem}\label{sqf/T}
Consider the polynomial ring $T:=\operatorname{Sym} A_1 \cong  \kk[t_1, \ldots, t_n]$   
(note that $T$ is {\it not}  a subring of $A$).  
Then $A$ is a squarefree $T$-module. 
\end{lem}

\begin{rem}\label{quick proof}
Since $A$ is a squarefree $T$-module, Duval's formula on $H_\m^i(A)$ 
immediately follows from \cite[Lemma~2.9]{Y}. However, 
since $H_\m^i(A)$ has a finer ``grading" (see \cite{Du} or Corollary~\ref{-sqf} below), 
the formula will be mentioned in Corollary~\ref{reduced cohomology}.   
\end{rem}

\noindent{\it The proof of Lemma~\ref{positive}.}
Let $T$ be as in Lemma~\ref{sqf/T}. 
For  $\11 := (1,1, \ldots,1) \in \NN^n$, $T(-\11)$ is the ($\ZZ^n$-graded) canonical 
module of $T$. By the local duality theorem, we have    
$$H^i(\cpx{{}^*\!D}_A) \cong \uExt^i_A(A, \cpx{{}^*\!D}_A) \cong \uExt^{n+i}_T(A, T(-\11)).$$
By \cite[Theorem~2.6]{Y}, $\uExt^{n+i}_T(A, T(-\11))$ is a squarefree module, in particular, 
$\NN^n$-graded.  
\qed

\bigskip

\noindent{\it The proof of Theorem~\ref{main}.}
Recall the cochain map $\phi^\cdot: \cpx I_A \to \cpx{\D*}_A$ constructed 
before Lemma~\ref{positive}.   
By \eqref{phi positive}, 
$\phi^\cdot$ gives the isomorphism $\cpx I_A \cong (\cpx{\D*}_A)_{\geq \b0}$. 
Hence $\phi^\cdot$ is a quasi-isomorphism by Lemma~\ref{positive}. 
Since  $\cpx{\D*}_A \cong \cpx D_A$ in $\Db(\Mod A)$, we are done. 
\qed

\section{Squarefree Modules over $A_P$}
In this section, we will define a squarefree module over the face ring $A =A_P$ of a simplicial poset $P$.   
For this purpose, we equip $A$ with a finer ``grading", where the index set is no longer a monoid  
(similar idea has appeared in \cite{Du,OY}).

Recall the convention that $\{ \, y \in P \mid \rho(y) = 1 \, \}=\{ y_1, \ldots, y_n \}$ 
and $t_i = t_{y_i} \in A$. 
For each $x \in P$, set $$\M(x):= \bigoplus_{y_i \leq x}\NN \, \be^x_i,$$
where $\be_i^x$ is a basis element. So $\M(x) \cong \NN^{\rho(x)}$ as additive monoids. 
For $x,z$ with $x \leq z$, we have an injection 
$\iota_{z,x} : \M(x) \ni \be_i^x \mapsto \be_i^z \in \M(z)$ of monoids. 
Set $$\M := \colimit_{x \in P}\M(x),$$
where the direct limit is taken in the category of sets 
with respect to  $\i_{z, x} : \M(x) \to \M(z)$ for $x, z \in P$ with $x \leq  z$. 
Note that $\M$ is no longer a monoid. 
Since all $\i_{z,x}$  is injective, we can regard $\M(x)$ as a subset of $\M$. 
For each $\ua \in \M$, $\{  x \in P \mid \ua \in \M(x)  \}$ 
has the smallest element, which is denoted by $\s(\ua)$. 

We say a monomial $$\sm = \prod_{x \in P} t_x^{n_x} \in A \qquad  (n_x \in \NN)$$ is {\it standard}, 
if $\{ \, x \in P \mid n_x \ne 0 \, \}$ is a chain. 
In this case, set $\s(\sm):= \max \{ \, x \in P \mid n_x \ne 0 \}$.
If $n_x = 0$ for all $x \ne \hat 0$, then $\m=1$. Hence 1 is a standard monomial with $\s(1)= \hat 0$. 
As shown in \cite{St91}, the set of standard monomials forms a $\kk$-basis of $A$.

There is a one-to-one correspondence between the elements of $\M$ and the standard 
monomials of $A$. For a standard monomial $\sm$, 
set  $U := \{ \, i \in [n] \mid y_i \leq \s(\sm) \, \}$. Then we have $\s(\sm) \in [U]$.
There is $\ba \in \NN^U$ such that the image of $\sm$ in $A/\p_{\s(\sm)} 
\cong \kk[\, t_i \mid i \in U \, ]$ 
is a monomial of the form $\prod_{i \in U } t_i^{a_i}$.  So $\sm$ corresponds to 
$\ua \in \M(\s(\sm)) \, (= \bigoplus_{i \in U} \NN \, \be^{\s(\sm)}_i)  \subset \M$ 
whose $e_i^{\s(\sm)}$-component is $a_i$.  We denote this $\sm$ by $t^\ua$. 

Let $\ua, \ub \in \M$. 
If $[\s(\ua) \vee \s(\ub)] \ne \emptyset$, then we can take the sum $\ua+\ub \in \M(x)$ 
for each $x \in [\s(\ua) \vee \s(\ub)]$. 
Unless $[\s(\ua) \vee \s(\ub)]$ consists of a single element, we cannot define 
$\ua +\ub \in \M$. Hence we denote each $\ua +\ub \in \M(x)$ by $(\ua+\ub)|x$. 

\begin{dfn}
$M \in \Mod A$ is said to be {\it $\M$-graded} if the following are satisfied;
  \begin{enumerate}
  \item $M = \Dsum_{\ua \in \M}M_\ua$ as $\kk$-vector spaces;
  \item For $\ua, \ub \in \M$, we have 
$$t^\ua M_{\ub} \subset \Dsum_{x \in [\s(\ua) \vee \s(\ub)]} M_{(\ua+\ub)|x}.$$
Hence, if $[\s(\ua) \vee \s(\ub)] = \emptyset$, then $t^\ua M_{\ub}=0$. 
\end{enumerate}
\end{dfn}
 
Clearly, $A$ itself is an  $\M$-graded module with $A_\ua = \kk \, t^\ua$. 
Since there is a natural map $\M \to \NN^n$, an $\M$-graded module can be seen as an $\NN^n$-graded module. 

If $M$ is an  $\M$-graded $A$-module, then 
$$M_{\not \leq x} :=\bigoplus_{\ua \not \in \M(x)} M_\ua$$ 
is an $\M$-graded submodule for all $x \in P$, and $$M_{\leq x} :=  M/M_{\not \leq x}$$ 
is a $\ZZ^{\rho(x)}$-graded module over  $A/\p_{x} \cong \kk[\, t_i \mid y_i \leq x \, ]$. 

 \begin{dfn}
 We say an $\M$-graded $A$-module $M$ is {\it squarefree}, if $M_{\leq x}$ is a squarefree module over 
 the polynomial ring $A/\p_{x} \cong \kk[\, t_i \mid y_i \leq x \, ]$ for all $x \in P$. 
 \end{dfn}

Note that squarefree $A$-modules are automatically finitely generated, 
and can be seen as squarefree modules over $T= \Sym A_1$.

Clearly, $A$ itself, $\p_x$ and $A/\p_x$ for $x \in P$, are squarefree. 
Let $\Sq A$ be the category of squarefree $A$-modules and their 
$A$-homomorphisms $f:M \to M'$ with $f(M_\ua) \subset M'_\ua$ 
for all $\ua \in \M$. For example, $\cpx I_A$ is a complex in $\Sq A$.  
To see basic properties of $\Sq A$, we introduce  
the {\it incidence algebra} of the poset $P$ as in \cite{Y05} 
(so consult \cite{Y05} for further information). 
 
The incidence algebra $\Lambda$ of $P$ over $\kk$ 
is a finite dimensional associative $\kk$-algebra with basis
$\set{e_{x,y}}{x, y \in P,  \, x \ge y}$ whose multiplication is defined by
$$e_{x,y}\cdot e_{z,w} =\delta_{y,z} \,e_{x,w},$$
where $\delta_{y,z}$ denotes Kronecker's delta. 

Set $e_x := e_{x,x}$ for $x \in P$. 
Each $e_x$ is an idempotent, and $\Lambda e_x$ is indecomposable as a left $\Lambda$-module.
Clearly,  $e_x \cdot e_y = 0$  for $x \not= y$, and that $1_A = \sum_{x \in P} e_x$. 
Let $\mod \Lambda$ be the category of finitely generated left $\Lambda$-modules. 
As a $\kk$-vector space, $N \in \mod \Lambda$ has the decomposition  $N = \Dsum_{x \in P} e_x N$.  
Henceforth we set $N_x := e_x N$. Clearly, $e_{x, y} N_y \subset N_x,$ 
and $e_{x,y} N_z =0$ if $y \ne z$. 

For each $x \in P$,  
we can construct a left $\Lambda$-module as follows:
Set
$$
E_{\Lambda}(x) := \Dsum_{y  \in P, \ y \le x}\kk \, \bar e_y,
$$
where $\bar e_{y}$'s are basis elements. 
The module structure of $E_\Lambda(x)$ is defined by
$$
e_{z, \, w} \cdot \bar e_y = \begin{cases}
\bar e_z & \text{if $w = y$ and $z \le x $;} \\
0 & \text{otherwise.}
\end{cases}
$$
Then $E_{\Lambda}(x)$ is indecomposable and injective in $\mod \Lambda$. Conversely,  
any indecomposable injective is of this form.    
Moreover, $\mod \Lambda$ is an abelian category with enough injectives, 
and the injective dimension of each object is at most $d$.

\begin{prop}\label{sec:Sq_cat}
There is an equivalence between $\Sq A$ and $\mod \Lambda$.
Hence $\Sq A$ is an abelian category with enough injectives and 
the injective dimension of each object is at most $d$. An object $M \in \Sq A$ is 
an indecomposable injective if and only if $M \cong A/\p_x$ for some $x \in P$.
\end{prop}

\begin{proof}
Let $N \in \mod \Lambda$. 
For each $\ua \in \M$, we assign a $\kk$-vector space $M_\ua$ with a bijection  
$\mu_\ua : N_{\s(\ua)} \to M_\ua$. 
We put  an $\M$-graded $A$-module structure on $M:= \bigoplus_{\ua \in \M} M_\ua$ by 
$$t^{\ua} \, s = \sum_{x \in [\s(\ua) \vee \s(\ub)]} \mu_{(\ua+\ub)|x} ( 
e_{x, \s(\ub)}  \cdot  \mu^{-1}_{\ub}(s) ) \quad 
\text{for $s \in M_\ub$.}$$
To see that $M$ is actually an $A$-module, note that both $(t^{\ua} \, t^{\ub}) \, s$ 
and $t^{\ua}\,  (t^{\ub} \, s)$ equal 
$$\sum_{x \in [\s(\ua) \vee \s(\ub) \vee \s(\uc)]} 
\mu_{(\ua+\ub+\uc )|x} ( 
e_{x, \s(\uc)}  \cdot  \mu^{-1}_{\uc}(s) ) \quad 
\text{for $s \in M_\uc$.}$$
We can also show that $M$ is squarefree. 

To construct the inverse correspondence, for $x \in P$ with $r = \rho(x)$, 
set $\ua(x) := (r,r, \ldots, r) \in \NN^r \cong \M(x) \subset \M$. 
If $x \geq y$, then there is a degree $\ua(x)-\ua(y) \in \M(x) \subset \M$ such that 
$t^{\ua(x)-\ua(y)} \cdot t^{\ua(y)} = t^{\ua(x)}$.  (One might think a simpler definition 
 $\ua(x) := (1,1, \ldots, 1) \in \NN^r$ works. However this is not true. 
In this case, the candidate of $\ua(x)-\ua(y)$ belongs to $\M(z)$ for some 
$z \in P$ with $z < x$.  So $(\ua(x)-\ua(y)) +\ua(y)$ does not exist,  
unless $\# [y \vee z] =1$.)
Now we can construct $N \in \mod \Lambda$ from $M \in \Sq A$ as follows:  
Set $N_x := M_{\ua(x)}$, and define the multiplication map $N_y \ni s \mapsto e_{x,y} \cdot s \in N_x$ 
by $M_{\ua(y)} \ni s \mapsto t^{\ua(x)-\ua(y)} s \in M_{\ua(x)}$ for $x, y \in P$ with $x \geq y$. 

By this correspondence, we have $\Sq A \cong \mod \Lambda$. 
For the last statement,  note that $E_\Lambda(x) \in \mod \Lambda$ corresponds to $A/\p_x \in \Sq A$.  
\end{proof}

Let $\InjSq$ be the full subcategory of $\Sq A$ consisting of all injective objects, that is, 
finite direct sums of copies of $A/\p_x$ for various $x \in P$. 
As is well-known, the bounded homotopy category $\Cb(\InjSq)$ is equivalent to $\Db(\Sq A)$.  
Since $$\uHom_A(A/\p_x, A/\p_y) = \begin{cases} 
A/\p_y & \text{if $x \geq y$,}\\
0 & \text{otherwise,}
\end{cases}$$
we have $\uHom_A^\bullet(\cpx J, \cpx I_A) \in \Cb(\InjSq)$ for all $\cpx J \in \Cb(\InjSq)$. 
Moreover, $\uHom_A^\bullet(-, \cpx I_A)$ 
gives a functor $$\DD: \Cb(\InjSq) \to \Cb(\InjSq)^\op.$$ 

\begin{prop}
Via the forgetful functor  $\for: \InjSq \to \gr A$,  $\DD$ coincides with $\RuHom_A (-, \cpx{\D*}_A)$. 
More precisely, we have a natural isomorphism 
$$\Phi: \for \circ \DD \stackrel{\cong}{\longrightarrow} \RuHom_A (-, \cpx{\D*}_A) \circ \for.$$ 
Here both $\for \circ  \DD$ and $\RuHom_A (-, \cpx{\D*}_A) \circ \for$ are functors from 
$\Cb(\InjSq)$ to $\Db(\gr A)$.  
\end{prop}

\begin{proof}
The cochain map $\phi^\bullet: \cpx I_A \to \cpx {\D*}_A$ induces the natural transformation 
$\Phi$. It remains to prove that 
$\Phi(\cpx J): \DD(\cpx J) \to  \RuHom_A (\cpx J, \cpx{\D*}_A)$ is a quasi isomorphism 
for all $\cpx J \in \Cb(\InjSq)$. 
For this fact, we use a similar argument to the final steps of the previous section 
(while the the same argument as the proof of \cite[Proposition~5.4]{OY} also works). 
Note that $\cpx J$ is a complex of squarefree modules over the polynomial ring $T:= \Sym A_1$. 
Since $\RuHom_A (\cpx J, \cpx{\D*}_A) \cong \RuHom_T (\cpx J, \cpx{\D*}_T)$ by the local duality theorem,   
the cohomologies of $\RuHom_A (\cpx J, \cpx{\D*}_A)$ are squarefree $T$-modules, 
in particular, $\NN^n$-graded. On the other hand, through $\Phi$, $\DD(\cpx J)$ 
is isomorphic to the $\NN^n$-graded part of $\RuHom_A (\cpx J, \cpx{\D*}_A)$.  
\end{proof}

\begin{rem}
By the equivalence $\Cb(\InjSq) \cong \Db(\Sq A)$, $\DD$ can be regarded as 
a contravariant functor from $\Db(\Sq A)$ to itself.  
Then,  through the equivalence $\Sq R \cong \mod \Lambda$, $\DD$ coincides with the functor 
$\DDD: \Db(\mod \Lambda) \to \Db(\mod \Lambda)^\op$ defined in \cite{Y05} 
up to translation. 
Hence, for $\cpx M \in \Db(\Sq A)$, the complex $\DD(\cpx M)$ has the following description:  
The term of cohomological degree $p$ is 
$$
\DD(\cpx M)^p := \Dsum_{i + \rho(x) = -p} (M^i_{\ua(x)})^* \tns_\kk A/\p_x,
$$
where $(-)^*$ denotes the $\kk$-dual, and $\ua(x) \in \M(x) \subset \M$ is the one defined in the proof of 
Proposition~\ref{sec:Sq_cat}. 
The differential is given by 
$$
 (M^i_{\ua(x)})^* \tns_\kk A/\p_x \ni f \otimes 1_{A/\p_x} \longmapsto 
\sum_{\substack{y \le x,\\ \rho(y)= \rho(x) - 1}} 
\epsilon(x,y) \cdot f_y \tns 1_{A/\p_y} + 
(-1)^p \cdot f \circ \partial^{i-1}_{\cpx M} \tns 1_{A/\p_x}, 
$$
where $f_y \in (M_{\ua(y)})^*$ denotes 
$M_{\ua(y)} \ni s \mapsto f(t^{\ua(x)-\ua(y)} \cdot s) \in \kk$, and $\epsilon(x,y)$ 
is the incidence function.
We also have $\DD \circ \DD \cong \idmap_{\Db(\Sq A)}$. 
\end{rem}

Since $H^{-i}(\DD(M)) \cong \uExt_A^{-i}(M, \cpx{\D*}_A) \cong H_\m^i(M)^\vee$ in $\Gr A$, 
we have the following. 

\begin{cor}\label{-sqf}
If $M \in \Sq A$, then 
the local cohomology $H_\m^i(M)^\vee$ can be seen as a squarefree $T$-module. 
\end{cor}

\section{Sheaves and Poincar\'e-Verdier duality}
The results in this section are parallel to those in \cite[Section 6]{OY} 
(or earlier work \cite{Y03}). 
Although the relation between the rings treated there and our $A_P$ is not so direct, 
the argument is very similar. So we omit the detail of some proofs here.

Recall that a simplicial poset $P$ gives a regular cell complex $\Gamma(P)$.   
Let $X$ be the underlying space of $\Gamma(P)$, and $c(x)$ the open cell 
corresponding to $\hat{0} \ne x \in P$. 
Hence, for each $x \in P$ with $\rho(x) \geq 2$, $c(x)$ is an open subset of $X$ homeomorphic to 
$\RR^{\rho(x)-1}$ (if $\rho(x)=1$, then $c(x)$ is a single point), 
and $X$ is the disjoint union of the cells $c(x)$. 
Moreover, $x \geq y$ if and only if $\overline{c(x)} \supset c(y)$.  

As in the preceding section, let $\Lambda$ be the incidence algebra of $P$, 
and $\mod \Lambda$ the category of finitely generated left $\Lambda$-modules.   
In \cite{Y05}, we assigned the constructible sheaf $N^\dagger$ on $X$ to $N \in \mod \Lambda$.  
By the equivalence $\Sq A \cong \mod \Lambda$, we have the constructible sheaf $M^+$ on $X$ 
corresponding to $M \in \Sq A$. 
Here we give a precise construction for the reader's convenience.   
For the sheaf theory, consult \cite{Iver}.

For $M \in \Sq A$, set $$\sp(M) 
:= \bigcup_{\hat{0} \ne x \in P} c(x) \times M_{\ua(x)},$$
where $\ua(x) \in \M(x) \subset \M$ is the 
one defined in the proof of Proposition~\ref{sec:Sq_cat}. 
Let $\pi : \sp(M) \to X$ be the projection map which sends $(p, m) \in 
c(x) \times M_{\ua(x)} \subset \sp(M)$ to $p \in c(x) \subset X$. 
For an open subset $U \subset X$ and a map $s: U \to \sp(M)$, 
we will consider the following conditions:

\begin{itemize}
\item[$(*)$]  $\pi \circ s = \idmap_{U}$ and $s_p = t^{\ua(x)-\ua(y)} 
\cdot s_q$ for all $p \in c(x) \cap U$, $q \in c(y) \cap U$ with $ x \geq y$. 
Here $s_p \in M_{\ua(x)}$ (resp. $s_q \in M_{\ua(y)}$) is an element with  
$s(p) = (p, s_p)$ (resp.  $s(q) = (q, s_q)$).  
\item[$(**)$] There is an open covering $U = \bigcup_{i \in I} 
U_i$ such that the restriction of $s$ to $U_i$ satisfies $(*)$ for all $i \in I$. 
\end{itemize}

Now we define a sheaf $M^+$ on $X$ as follows:
For an open set $U \subset X$, set 
$$M^+(U):= 
\{ \, s \mid \text{$s: U \to \sp(M)$ is a map satisfying $(**)$} \,\}$$
and the restriction map $M^+(U) \to M^+(V)$ for $U \supset V$ is the natural one. 
It is easy to see that $M^+$ is a constructible sheaf with respect to the cell decomposition 
$\Gamma(P)$. For example, $A^+$ is the $\kk$-constant sheaf $\const_X$ on $X$, 
and $(A/\p_x)^+$ is (the extension to $X$ of) the $\kk$-constant sheaf on the closed cell 
$\overline{c(x)}$.

Let $\Sh(X)$ be the category of sheaves of finite dimensional 
$\kk$-vector spaces on $X$. The functor $(-)^+: \Sq A \to \Sh(X)$ is exact. 

As mentioned in the previous section, 
$\DD:\Db(\Sq A) \to \Db(\Sq A)^\op$ corresponds to 
$\TT \circ \mathbf D: \Db(\mod \Lambda) \to \Db(\mod \Lambda)^\op$, 
where $\mathbf D$ is the one defined in \cite{Y05}, and $\TT$ 
is the translation functor (i.e., $\TT(\cpx M)^i = M^{i+1}$). 
Through $(-)^\dagger: \mod \Lambda \to \Sh(X)$, 
$\DDD$ gives the Poincar\'e-Verdier duality on $\Db(\Sh(X))$, 
so we have the following.

\begin{thm}\label{Verdier}
For $\cpx M \in \Db(\Sq A)$, we have 
$$\TT^{-1} \circ \DD(\cpx M)^+  \cong \RcHom((\cpx M)^+, \cDx)$$
in $\Db(\Sh(X))$.  In particular, $\TT^{-1}((\cpx I_A)^+) \cong \cDx$, where 
$\cpx I_A$ is the complex constructed in Theorem~\ref{main}, and $\cDx$ is 
the  Verdier dualizing complex of $X$ with the coefficients in $\kk$. 
\end{thm}

The next result follows from results in \cite{Y05} 
(see also  \cite[Theorem~6.2]{OY}). 

\begin{thm}\label{local cohomology}
For $M \in \Sq A$, we have the decomposition $H_\m^i(M) 
= \Dsum_{\ua \in \M} H_\m^i(M)_{-\ua}$ by Corollary~\ref{-sqf}. 
Note that $\M$ has the element $\b0$. Then the following hold.  
\begin{itemize}
\item[(a)] There is an isomorphism  
$$H^i(X, M^+) \cong H_\m^{i+1}(M)_{\b0} \quad  \text{for all $i \geq 1$},$$ 
and an exact sequence 
$$0 \to H_\m^0(M)_{\b0} \to M_{\b0} \to H^0( X, M^+) \to H_\m^1(M)_{\b0} \to 0.$$
\item[(b)] If $\b0 \ne \ua \in \M$ with $x = \s(\ua)$, then 
$$H_\m^i(M)_{-\ua} \cong H^{i-1}_c(U_x, M^+|_{U_x})$$
for all $i \geq 0$. Here $U_x = \bigcup_{z \geq x} c(z)$ is an open set of $X$, and 
$H^\bullet_c(-)$ stands for the cohomology with compact support. 
\end{itemize}
\end{thm}

Let  $\rH^i(X;\kk)$ denote  the $i$th {\it reduced cohomology} of 
$X$ with coefficients in $\kk$. That is, $\rH^i(X;\kk) \cong H^i(X;\kk)$ for all $i \geq 1$, and 
 $\rH^0(X;\kk) \oplus \kk \cong H^0(X;\kk)$, where $H^i(X;\kk)$  is  the usual 
cohomology of $X$. 

\begin{cor}[Duval {\cite[Theorem~5.9]{Du}}]\label{reduced cohomology}
We have $$[H_\m^i(A)]_{\b0} \cong \rH^{i-1}(X; \kk) \quad  \text{and} \quad 
[H_\m^i(A)]_{-\ua} \cong H^{i-1}_c(U_x;\kk)$$ for all $i \geq 0$
and all $\b0 \ne \ua \in \M$ with $x = \s(\ua)$.  

For this $\ba \in \M$ (but $\ba$ can be $\b0$ here), 
$[H_\m^i(A)]_{-\ua}$ is also isomorphic to the $i$th cohomology of the cochain complex 
$$K_x^\bullet: 0 \to K_x^{\rho(x)} \to K_x^{\rho(x)+1} \to \cdots \to K_x^d \to 0 \quad \text{with} \quad 
K_x^i = \Dsum_{\substack{z \geq x \\ \rho(z)=i}} \kk \,b_z$$
($b_z$ is a basis element) whose differential map is given by 
$$b_z \longmapsto \sum_{\substack{w \geq z \\ \rho(w)=\rho(z)+1}} \epsilon(w,z) \, b_w.$$
\end{cor}

Duval uses the latter description, and he denotes $H^i(\cpx K_x)$ by 
$H^{i-\rho(x)-1}(\operatorname{lk} x)$.

\begin{proof}
The former half follows from Theorem~\ref{local cohomology} by the same argument as
\cite[Corollary~6.3]{OY}. The latter part follows from that $H_\m^i(A) \cong H^{-i}(\DD(A))^\vee$ and 
that $(\DD(A)^\vee)_{-\ua} = \cpx K_x$ as complexes of $\kk$-vector spaces. 
\end{proof}

\begin{thm}[c.f. Duval \cite{Du}]\label{CM & Gor} 
Set $d:=\rank P = \dim X +1$. Then we have the following. 
\begin{itemize}
\item[(a)]  
$A$ is Cohen-Macaulay if and only if $\cH^i(\cDx) = 0$ for all $i \ne -d+1$, and  
$\rH^i(X;\kk)=0$ for all $i \ne d-1$. 
\item[(b)] Assume that $A$ is Cohen-Macaulay and $d \geq 2$. Then 
$A$ is Gorenstein*, if and only if $\cH^{-d+1}(\cDx) \cong \const_X$.  
(When $d =1$, $A$ is Gorenstein* if and only if $X$ consists of exactly two points.)  
\item[(c)] $A$ is Buchsbaum if and only if  $\cH^i(\cDx) = 0$ for all $i \ne -d+1$. 
\item[(d)] Set $$d_i := \begin{cases} \dim (\supp 
\cH^{-i} (\cDx)) & \text{if $\cH^{-i} (\cDx) \ne 0$,}\\
-1 & \text{if $\cH^{-i} (\cDx) = 0$ and $\rH^i(X;\kk) \ne 0$,}\\
-\infty & \text{if $\cH^{-i} (\cDx) = 0$ and $\rH^i(X;\kk) = 0$.}
\end{cases}
$$ 
Here $\supp \cF = \{ \, p \in X \mid \cF_p \ne 0 \, \}$ for a sheaf $\cF$ on $X$.
Then, for $r \geq 2$, $A$ satisfies Serre's condition  
$(S_r)$ if and only if $d_i \leq i-r$ for all $i < d-1$.  
\end{itemize}
Hence, Cohen-Macaulay (resp. Gorenstein*, Buchsbaum) property 
and Serre's condition $(S_r)$ of $A$ are
topological properties  of $X$, while they may depend on $\chara(\kk)$. 
\end{thm}

As far as the author knows, even in the Stanley-Reisner ring case, 
(d) has not been mentioned in literature yet. 

Recall that we say $A$ satisfies Serre's condition $(S_r)$ if 
$\depth A_\p \geq \min\{ \, r, \, \hght \p \, \}$ for all prime ideal $\p$ of $A$.  
The next fact is well-known to the specialist, but we will sketch the proof here 
for the reader's convenience. 

\begin{lem}\label{S_r iikae}
For $r \geq 2$, $A$ satisfies the condition $(S_r)$ if and only if 
$\dim H^{-i} (\cpx I_A)\leq i-r$ for all $i < d$.  
Here the dimension of the 0 module is $-\infty$. 
\end{lem}

\begin{proof}
For a prime ideal $\p$, the normalized dualizing complex of $A_\p$ is quasi-isomorphic to 
$\TT^{-\dim A/\p}(\cpx I_A \otimes_A A_\p).$
Hence we have 
\begin{equation}\label{depth form}
\depth A_\p = \min \{ \, i \mid (H^{-i}(\cpx I_A) \otimes_A A_\p) \ne 0 \, \}-\dim A/\p.
\end{equation} 
Recall that if $A$ satisfies Serre's condition $(S_2)$ 
then $P$ is {\it pure} (equivalently, $\dim A/\p = d$ for all minimal prime ideal $\p$). 
Similarly, if $\dim H^{-i} (\cpx I_A) <i$ for all $i < d$, 
then $P$ is pure. (In fact, if $\p$ is a minimal prime ideal of $A$ 
with $i:= \dim A/\p <d$, then $\depth A_\p = 0$ implies that $\p$ is a minimal prime 
of $H^{-i}(\cpx I_A)$. It follows that $\dim H^{-i}(\cpx I_A)=i$. 
This is a contradiction.)
So we may assume that $P$ is pure, and hence $\dim A/\p + \hght \p =d$ for all $\p$. 
Now the assertion follows from \eqref{depth form}.    
\end{proof}

\noindent{\it The proof of Theorem~\ref{CM & Gor}.}
We can prove (a)--(c) by the same way as \cite[Theorems~6.4 and 6.7]{OY}.
For (d), note that $d_j = \dim H^{-j-1}(\cpx I_A) -1$. 
So the assertion follows from Lemma~\ref{S_r iikae}. 
\qed



\section{Further discussion}
This section is a collection of miscellaneous results related to the arguments in the previous sections.  

\medskip

For an integer $i \leq d-1$, the poset 
$$P^{(i)} := \{ \, x \in P \mid \rho(x) \leq i+1 \, \}$$
is called the $i$-{\it skeleton} of $P$. Clearly, $P^{(i)}$ is simplicial again, and set $A^{(i)} := A_{P^{(i)}}$.  
Then it is easy to see that the (Theorem~\ref{main} type) dualizing complex 
$\cpx I_{A^{(i)}}$ of $A^{(i)}$ coincides with the brutal truncation 
$$0 \to I_A^{-i-1} \to I_A^{-i} \to \cdots \to I_A^0 \to 0$$
of $\cpx I_A$. 
Since $\depth A=  \min \{ \, i \mid H^{-i}(\cpx I_A) \ne 0 \, \}$ and 
$\dim A =   \max \{ \, i \mid H^{-i}(\cpx I_A) \ne 0  \, \}$,  
we have the equation  
\begin{equation}\label{skeleton}
\depth A_P = 1+ \max \{\, i \mid \text{$A^{(i)}$ is Cohen-Macaulay}\, \}, 
\end{equation}
which is \cite[Corollary~6.5]{Du}.

\medskip

Contrary to the Gorenstein* property, the Gorenstein property of $A_P$ is {\it not} topological. 
This phenomena occurs even for Stanley-Reisner rings. 
But there is a characterization of $P$ such that $A_P$ is Gorenstein. 
For posets $P_1, P_2$, we regard $P_1 \times P_2 = \{ \, (x_1, x_2) \mid 
x_1 \in P_1, \, x_2 \in P_2 \, \}$ as a poset by 
$(x_1, x_2) \geq (y_1, y_2) \stackrel{\text{def}}{\Longleftrightarrow}$ 
$x_i \geq y_i$ in $P_i$ for each $i=1,2$.

\begin{prop}\label{Gorenstein}
$A_P$ is Gorenstein if and only if $P \cong 2^V \times Q$ as posets for 
a boolean algebra $2^V$ and a simplicial poset $Q$ with $A_Q$ is Gorenstein*. 
\end{prop}

\begin{proof}
The  sufficiency is clear. In fact, if $P \cong 2^V \times Q$ then $A:=A_P$ is 
a polynomial ring over $A_Q$. So it remains to prove the necessity.  

By Lemma~\ref{sqf/T}, $A$ is a squarefree module over the polynomial ring 
$T := \operatorname{Sym} A_1$. 
We say $\ba=(a_1, \ldots, a_n) \in \NN^n$ is {\it squarefree}, if $a_i =0,1$ for all $i$. 
If this is the case, we identify $\ba$ with its support $\{ \, i \mid a_i = 1 \, \}\subset [n]$. 
Hence a subset $F \subset [n]$ sometimes means the corresponding squarefree vector 
in $\NN^n$.  

Since $A$ is Gorenstein (in particular, Cohen-Macaulay) now, 
a minimal $\ZZ^n$-graded $T$-free resolution of $A$ is of the form 
$$L_\bullet: 0 \to L_{n-d} \to \cdots \to L_1 \to L_0 \to 0 \quad 
\text{with} \quad L_i = \bigoplus_{F \subset [n]}T(-F)^{\beta_{i,F}}$$
by \cite[Corollary~2.4]{Y}.

Let  $\11 := (1,1, \ldots,1) \in \NN^n$.  Note that $\uHom_T^\bullet(L_\bullet, T(-\11))$ is 
a minimal $\ZZ^n$-graded $T$-free resolution of the canonical module 
$\omega_A =\uExt_T^{n-d}(A,T(-\11))$ of $A$ up to translation, 
and $\omega_A \cong A(-V)$ for some 
$V \subset [n]$.  Set $W := [n] \setminus V$. 
Since $\uHom_T(T(-F), T(-\11)) \cong T(-([n] \setminus F))$, 
we have the following:

\medskip

$(*)$ If $\beta_{i, F} \ne 0$ for some $i$, then $F \subset W$.  

\medskip

If $[V] = [\bigvee_{i \in V}y_i] = \emptyset$, then 
by the construction of $A$, there is some $F \subset V$ with $\beta_{1,F} \ne 0$, 
and it contradicts to the statement $(*)$. 
Even if $\# [V] \geq 2$, the same contradiction occurs. 
Hence $[V] = \{ x \}$ for some $x \in P$. 
We denote the closed interval $[\hat{0},x]$ by $2^V$. 

Set 
$$Q:= \{ \, z \in P \mid \text{$z \not \geq y_i$ for all $i \in V$} \, \}  
= \coprod_{U \subset W} [U].$$  
If $\# [x' \vee z] \ne 1$ for some $x' \in 2^V$ and $z \in Q$, 
then $\beta_{1,F} \ne 0$ for some $F$ with 
$F \cap V \ne \emptyset$, and it contradicts to $(*)$. 
Hence, for all $x' \in 2^V$ and $z \in Q$, we have $\# [x' \vee z] =1$. 
Denoting the element of $[x' \vee z]$ by $x' \vee z$,  
we have an order preserving map 
$$\psi: 2^V \times Q \ni (x', z) \longmapsto x' \vee z \in P.$$ 
Moreover, since $P = \coprod_{U \subset[n]} [U]$, $\psi$ is an isomorphism of posets, 
and we have
$$P \cong 2^V \times Q.$$ 

Clearly, $Q$ is a simplicial poset. Set $B:=A_Q$. 
Since $A \cong B[\, t_i \mid i \in V \, ]$ and  
$$A(-V) \cong \omega_A \cong (\omega_B[ \, t_i \mid i \in V \,])(-V),$$
$B$ is Gorenstein*. 
\end{proof}

Let $\Sigma$ be a finite regular cell complex with the underlying topological space $X(\Sigma)$, 
and  $Y,Z \subset X(\Sigma)$ closed subsets with $Y \supset Z \ne \emptyset$. 
Set $U := Y \setminus Z$, and let $h: U \embto Y$ be the embedding map. 
We can define the Cohen-Macaulay property of the pair $(Y,Z)$, which generalizes 
the Cohen-Macaulay property of a relative simplicial complex introduced in \cite[III. \S 7]{St}. 
See Lemma~\ref{relative CM lem} below. 

\begin{dfn}\label{relative CM def}
We say the pair $(Y,Z)$ is {\it Cohen-Macaulay} (over $\kk$), if $H_c^i(U;\kk)=0$ for all 
$i \ne \dim U$ and ${\rm R}^ih_* \cDu=0$ for all $i \ne -\dim U$. Here $\cDu$ is 
the  Verdier dualizing complex of $U$ with the coefficients in $\kk$. 
\end{dfn}



We say an ideal $I \subset A$ is {\it squarefree}, if it is generated by a subset of  
$\{\, t_x \mid x \in P \, \}$. Clearly, an ideal $I$ is a squarefree 
submodule of $A$ if and only if it is a squarefree ideal. 
For a  squarefree ideal $I$, $\s(I):= \{ \, x \in P \mid t_x \in I \, \}$ is 
an order filter (i.e., $x \in \s(I)$ and $y \geq x$ imply $y \in \s(I)$), and 
$U_I:= \bigcup_{x \in \s(I)} c(x)$ is an open set of $X$.  
The sheaf $I^+$ is (the extension to $X$ of) the $\kk$-constant sheaf on $U_I$.

\begin{prop}\label{relative CM lem} 
(1) A squarefree ideal $I$ with $I \subsetneq \m$ is a Cohen-Macaulay module 
if and only if  $(\overline{U_I}, \, \overline{U_I} \setminus U_I)$ is 
Cohen-Macaulay in the sense of Definition~\ref{relative CM def}. 

(2) The {\it sequentially Cohen-Macaulay} (see \cite[III. Definition~2.9]{St}) 
property of $A$ depends only on $X$ (and $\chara (\kk)$). 
\end{prop}

\begin{proof}
(1) Set $U:= U_I$, and let $h: U \to \overline{U}$ be the embedding map. 
The assertion follows from the fact that 
$\TT^{-1} (\DD(I)^+ |_{\overline{U}}) \cong {\rm R}h_* \cDu$ and 
$[H^{-i}(\DD(I))]_{\b0} \cong H_c^{i-1}(U;\kk)$ by \cite{Y05} 
(see also \cite[Proposition~4.10]{Y03}). 

(2) Follows from (1) by the same argument as \cite[Theorem~4.7]{Y08}. 
\end{proof}

\begin{rem}
While it is not stated in \cite{OY}, the statements corresponding to 
Lemma~\ref{CM & Gor}, the equation \eqref{skeleton}
and Proposition~\ref{relative CM lem} hold for a cone-wise normal toric face ring. 
\end{rem}

As in \cite{Y05}, we regard the finite regular cell complex $\Sigma$ as a poset by 
$\sigma \geq \tau \stackrel{\text{def}}{\Longleftrightarrow} \overline{\sigma} \supset \tau$. 
Here we use the convention  that $\emptyset \in \Sigma$. 
We say $\Sigma$ is a {\it meet-semilattice} (or, satisfies the {\it intersection property}), 
if the largest common lower bound $\s \wedge \t \in \Sigma$ exists for all $\s,\t \in \Sigma$. 
It is easy to see that $\Sigma$ is a meet-semilattice if and only if $\# [\s \vee \t] \leq 1$ 
for all $\s, \t \in \Sigma$. 
The underlying cell complex of a toric face ring (especially,  
a simplicial complex) is a meet-semilattice.  

For $\s \in \Sigma$, let $U_\s$ be the open subset $\bigcup_{\tau \geq \sigma} \tau$ of 
$X(\Sigma)$.
As shown in \cite{Y05}, if $X(\Sigma)$ is Cohen-Macaulay and $\Sigma$ is a meet-semilattice, 
then $(\overline{U_\s}, \, \overline{U_\s} \setminus U_\s)$ is Cohen-Macaulay for all $\s$. 
(If $\Sigma$ is not a meet-semilattice, we have an easy counter example.)  
While a simplicial poset $P$ is {\it not} a meet-semilattice in general, 
the above fact remains true. 
We also remark that an indecomposable projective in $\Sq A$ is isomorphic to the ideal 
$J_x:=(t_x) \subset A$ for some $x \in P$.

\begin{prop}\label{relative CM}  
If $A$ is Cohen-Macaulay (resp. Buchsbaum) then the ideal $J_x := (t_x)$ is 
a Cohen-Macaulay module for all $x \in P$ (resp. for all $\hat{0} \ne x \in P$). 
\end{prop}

\begin{proof}
Let $\ua \in \M$ with $\s(\ua)=y$.
With the notation of Proposition~\ref{reduced cohomology}, recall that  
$${\rm R}\Gamma_\m A \cong (\DD(A)^\vee)_{-\ua} \cong \cpx K_y.$$ 
 Similarly, we have 
$${\rm R}\Gamma_\m J_x \cong 
(\DD(J_x)^\vee)_{-\ua} \cong \bigoplus_{z \in [x \vee y]} \cpx K_z.$$
To see the second isomorphism, note that if $w \geq x,y$ then there exists 
a {\it unique} $z \in [x \vee y]$ such that $w \geq z$.  

If $A$ is Cohen-Macaulay (resp. Buchsbaum) then 
$H^i_\m(A)_{-\ua} \cong H^i({\rm R}\Gamma_\m A)_{-\ua} = 0$ 
for all $i < d$ and all $\ua \in \M$ (resp. $\b0 \ne \ua \in \M$). Hence we are done. 
\end{proof}

Regarding $A=A_P$ as a $\ZZ$-graded ring, we have
$$\sum_{i \geq 0} (\dim_\kk A_i) \cdot \lambda^i 
= \frac{h_0 + h_1 \lambda + \cdots +h_d \lambda^d}{(1-\lambda)^d},$$ 
for some integers $h_0, h_1, \ldots, h_d$ by \cite[Proposition~3.8]{St}. 
We call $(h_0, \ldots, h_d)$ the {\it $h$-vector} of $P$. 
The behavior of the $h$-vectors of simplicial complexes is 
an important subject of combinatorial commutative algebra.  
The $h$-vector of a simplicial poset has also been studied, 
and striking results are given in \cite{St, Mas}.  

Recently, Murai and Terai gave nice results on the $h$-vector of a simplicial 
complex $\Delta$ such that the Stanley-Reisner ring $\kk[\Delta]$ satisfies 
Serre's condition $(S_r)$. We show that one of them also holds for 
simplicial posets. 

\begin{thm}[{c.f. Murai-Terai \cite[Theorem~1.1]{MT}}]\label{MT S_r}
Let $P$ be a simplicial poset with the $h$-vector $(h_0, h_1, \ldots, h_s)$. 
If $A$ satisfies Serre's condition $(S_r)$ then $h_i \geq 0$ for all $i \leq r$. 
\end{thm}

\begin{proof}
By virtue of Lemma~\ref{sqf/T}, 
we can imitate the proof of \cite[Theorem~1.1]{MT}. 

Let $\Delta$ be a simplicial complex with the vertex set $[n]$, 
$S= \kk[x_1, \ldots, x_n]$ the polynomial ring, and $\kk[\Delta]=S/I_\Delta$ 
the Stanley-Reisner ring of $\Delta$.  
To prove our theorem, we replace $\kk[\Delta]$ and $S$ 
in their argument  by $A=A_P$ and $T= \Sym A_1$ respectively. 
In the former half of the proof, they deal $\kk[\Delta]$ as just a      
finitely generated (graded) $S$-module, 
and the argument is clearly applicable to our $A$ and $T$. 
The latter  half of their proof is based on the fact that 
$\uExt_S^i(\kk[\Delta], S(-\11))$ is a squarefree $S$-module of dimension at most $n-i-r$. 
Hence if  the following holds, the argument in \cite{MT} works in our case. 

\medskip

\noindent {\bf Claim.} $\uExt_T^i(A, T(-\11))$ is a squarefree $T$-module of 
dimension at most $n-i-r$. 

\medskip

The proof is easy.  In fact,  the squarefree-ness follows from Lemma~\ref{sqf/T} and \cite[Theorem~2.6]{Y}. 
Moreover, since $ \uExt_T^{i}(A, T(-\11)) \cong \uExt_A^{-n+i}(A, \cpx{\D*}_A) \cong H^{-n+i}(\cpx I_A)$ 
by the local duality,  we have $\uExt_T^i(A, T(-\11)) \leq n-i-r$ by Lemma~\ref{S_r iikae}. 
\end{proof}


\begin{thebibliography}{99}
\bibitem{Du} A.M. Duval, Free resolutions of simplicial posets,  
J. Algebra {\bf 188} (1997), 363--399. 

\bibitem{Ha} R. Hartshorne,
\term{Residues and duality},
Lecture notes in Mathematics {\bf 20}, Springer, 1966. 

\bibitem{Iver} B. Iversen,  {\it Cohomology of sheaves,} 
Springer-Verlag, 1986. 


\bibitem{Mas} M. Masuda, $h$-vectors of Gorenstein* simplicial posets, Adv. Math. {\bf 194} (2005), 332--344.

\bibitem{MP} M. Masuda, T. Panov, On the cohomology of torus manifolds, Osaka J. Math. {\bf 43} (2006) 711--746.

\bibitem{MS} E. Miller and B. Sturmfels, 
Combinatorial Commutative Algebra, 
Grad. Texts in Math., Vol. 227,  Springer, 2004. 

\bibitem{MT} S. Murai and N. Terai, 
$h$-vectors of simplicial complexes with Serre's conditions, 
Math. Res. Lett. {\bf 16} (2009), 1015--1028. 



\bibitem{OY} R. Okazaki and K. Yanagawa, Dualizing complex of a toric face ring, 
 Nagoya Math. J. {\bf 196} (2009), 87--116. 

\bibitem{St91} 
R. Stanley, $f$-vectors and $h$-vectors of simplicial posets, J. Pure Appl. Algebra {\bf 71} (1991), 319--331. 

\bibitem{St} R. Stanley,  
Combinatorics and commutative algebra, 2nd ed. Birkh\"auser 1996.   


\bibitem{Y} K. Yanagawa, 
Alexander duality for Stanley-Reisner rings and squarefree 
$\NN^n$-graded modules, J. Algebra 225 (2000), 630--645.  

\bibitem{Y03}
K. Yanagawa,
Stanley-Reisner rings, sheaves, and Poincar\'e-Verdier duality,
Math. Res. Lett. {\bf 10} (2003), 635--650.


\bibitem{Y05}
K. Yanagawa,
Dualizing complex of the incidence algebra of a finite regular cell complex,
Illinois J. Math. {\bf 49} (2005), 1221--1243. 

\bibitem{Y08}
K. Yanagawa, 
Notes on $C$-graded modules over an affine semigroup ring $K[C]$, 
Commun. Algebra {\bf 38} (2008), 3122 --3146. 
\end{thebibliography}
\end{document}